\documentclass{article}

\usepackage{proof,rotating,amsmath,amssymb,authblk,amsthm}
\usepackage[all,cmtip]{xy}

\newcommand{\SL}{\mathop{/}}
\newcommand{\BS}{\mathop{\backslash}}
\newcommand{\CC}{\boldsymbol{!}\mathbf{L}^{\!*}}
\newcommand{\CCu}{\boldsymbol{!}\mathbf{L}^{\!*}_{\SL}}
\newcommand{\Dbb}{\mathbf{Db}{\boldsymbol{!}\boldsymbol{?}}}
\newcommand{\Dbbr}{\mathbf{Db}{\boldsymbol{!}\boldsymbol{?}}_{\mathbf{b}}}
\newcommand{\Ac}{\mathcal{A}}
\newcommand{\Bx}{\mathcal{B}}
\newcommand{\Rc}{\mathcal{R}}

\newcommand{\GAc}[1]{{!}\Gamma_{\!{#1}}}
\newcommand{\Gc}{\mathcal{G}}

\newcommand{\PERM}{\mathrm{perm}}
\newcommand{\CONTR}{\mathrm{contr}}
\newcommand{\CUT}{\mathrm{cut}}

\newcommand{\Var}{\mathrm{Var}}

\newcommand{\Lu}{\mathbf{L}^{\!*}}
\newcommand{\Luu}{\mathbf{L}^{\!*}_{\SL}}

\newcommand{\PP}{\mathfrak{p}}

\newcommand{\BUSZK}{\mathrm{B}}
\newcommand{\TQ}{\widetilde{y}^\PP}
\newcommand{\TZ}{\widetilde{z}^\PP}
\newcommand{\TW}{\widetilde{w}^\PP}
\newcommand{\EASY}{\mathrm{E}}

\newcommand{\Ba}{\mathbf{a}^\PP}
\newcommand{\Bb}{\mathbf{b}^\PP}
\newcommand{\Bc}{\mathbf{c}^\PP}
\newcommand{\Be}{\mathbf{e}^\PP}
\newcommand{\Bf}{\mathbf{f}^\PP}

\makeatletter
\def\blfootnote{\xdef\@thefnmark{}\@footnotetext}
\makeatother

\newtheorem{theorem}{Theorem}
\newtheorem{lemma}{Lemma}
\theoremstyle{remark}
\newtheorem{example}{Example}
\newtheorem{remark}{Remark}

\begin{document}

\title{Undecidability of the Lambek Calculus with
a Relevant Modality}

\author{Max Kanovich}
\affil{{\small University College London; \newline
National Research University Higher School of Economics (Moscow)}}
\author{Stepan Kuznetsov}
\affil{{\small Steklov Mathematical Institute (Moscow)}}
\author{Andre Scedrov}
\affil{{\small University of Pennsylvania (Philadelphia); \newline
National Research University Higher School of Economics (Moscow)}}

\date{}

\maketitle

\blfootnote{The final publication (published in Proc. Formal Grammar 2015/2016,
LNCS vol.~9804, pp.~240--256) is available at Springer
via http://dx.doi.org/10.1007/978-3-662-53042-9\_14}

\begin{abstract}
Morrill and Valent{\'\i}n in the paper ``Computational coverage
of TLG: Nonlinearity'' considered an extension of the Lambek calculus
enriched by a so-called ``exponential'' modality. This modality
behaves in the ``relevant'' style, that is, it allows contraction and
permutation, but not weakening. Morrill and Valent{\'\i}n stated an open
problem whether this system is decidable. Here we show its
undecidability.
Our result remains valid if we consider the fragment
where all division
operations have one direction. We also show that the derivability
problem in a restricted case, where the modality can be applied
only to variables (primitive types), is decidable and belongs
to the NP class.
\end{abstract}

\section{The Lambek Calculus Extended by a Relevant Modality}

We start with the version of the Lambek calculus, $\Lu$, that allows
empty left-hand sides of sequents (introduced in~\cite{Lambek1961}). We will introduce $\CC$---an extension of $\Lu$ with one modality,
denoted by ${!}$.

Formulae of $\CC$ are built from a set of variables ($\Var = \{ p, q, r, \ldots \}$) using two
binary connectives, $\SL$ {\em (right division)} and $\BS$ {\em (left division)}, and
additionally one unary connective, ${!}$. Capital Latin letters denote formulae;
capital Greek letters denote finite (possibly empty) {\em linearly ordered sequences} of formulae.

Following the linguistic tradition, 
formulae of the Lambek calculus (and its extensions) are also called {\em types.} In this
terminology, variables are called {\em primitive types.}

We present $\CC$ in the form of sequent calculus. Sequents of $\CC$ are of the form
$\Pi \to A$, where $A$ is a formula and $\Pi$ is a finite (possibly empty) linearly ordered sequence of
formulae. $\Pi$ and $A$ are called the antecedent and the succedent respectively.

The axioms and rules of $\CC$ are as follows: 

$$
\infer{A \to A}{}
$$

$$
\infer[(\SL\to)]{\Delta_1, B \SL A, \Gamma, \Delta_2 \to C}{\Gamma \to A & \Delta_1, B, \Delta_2 \to C}
\qquad
\infer[(\to\SL)]{\Gamma \to B \SL A}{\Gamma, A \to B}
$$

$$
\infer[(\BS\to)]{\Delta_1, \Gamma, A \BS B, \Delta_2 \to C}{\Gamma \to A & \Delta_1, B, \Delta_2 \to C}
\qquad
\infer[(\to\BS)]{\Gamma \to A \BS B}{A, \Gamma \to B}
$$

$$
\infer[({!}\to)]{\Gamma_1, {!}A, \Gamma_2 \to C}{\Gamma_1, A, \Gamma_2 \to C}
\qquad
\infer[(\to{!})]{{!}A_1, \ldots, {!}A_n \to {!}B}{{!}A_1, \dots, {!}A_n \to B}
$$

$$
\infer[(\PERM_1)]{\Delta_1, \Gamma, {!}A, \Delta_2 \to C}{\Delta_1, {!}A, \Gamma, \Delta_2 \to C}
\qquad
\infer[(\PERM_2)]{\Delta_1, {!}A, \Gamma, \Delta_2 \to C}{\Delta_1, \Gamma, {!}A, \Delta_2 \to C}
$$

$$
\infer[(\CONTR)]{\Delta_1, {!}A, \Delta_2 \to C}{\Delta_1, {!}A, {!}A, \Delta_2 \to C}
$$

We call ${!}$ the {\em relevant modality}, since it behaves in a relevant logic style,
allowing contraction and permutation, but not weakening. Recall that in the original Lambek
calculus there is neither contraction, nor permutation, nor weakening. The modality
is introduced to restore contraction and permutation in a controlled way.

The cut rule is not officially included in $\CC$. Morrill and Valent\'{\i}n~\cite{MorVal} claim
that it is admissible and that this fact can be proved using the standard procedure
(cf.~\cite{LMSS}). In this paper we consider the system without cut and don't need its
admissibility.

We also consider fragments of $\CC$. 
Since there is no cut rule in this system, it enjoys the
subformula property, and therefore if we restrict the set of connectives, we obtain
conservative fragments of $\CC$: $\CCu$ (where we have only $\SL$ and ${!}$), 
$\Lu$ (this is the ``pure'' Lambek calculus without
${!}$), $\Luu$. 

As we discuss in more detail in~\cite{KanKuzSce}, $\Lu$ can be considered~\cite{Abrusci90}\cite{Yetter90} as a fragment of
non-commutative variant of Girard's linear logic~\cite{Girard87}.
Our modality, ${!}$, follows the spirit of the
exponential connective in linear logic, allowing contraction
and permutation. However, in contrast with the linear logic case,
we don't allow weakening. On the other hand, as we discuss
in Section~\ref{S:ling}, our ${!}$ is motivated from the
linguistical point of view.

\begin{theorem}\label{Th:main}
The derivability problem for $\CC$ is undecidable. Moreover, the 
derivability problem  is undecidable even for $\CCu$.
\end{theorem}

\begin{remark}
$\CC$ has been constructed as a conservative fragment of a larger system $\Dbb$, introduced
in~\cite{MorVal}. Thus Theorem~\ref{Th:main}
provides undecidability of $\Dbb$ (solving an open question raised in~\cite{MorVal}).
\end{remark}


\section{Linguistic Examples and Motivations}\label{S:ling}

In this section we start from the standard examples
of Lambek-style syntactic analysis~\cite{Lambek1958}\cite{Carpenter} and then
follow~\cite{MorVal}.

In syntactic formalisms based on the Lambek calculus and
its variants, Lambek types (formulae) denote syntactic categories.
We use the
following standard primitive types: $n$ stands for {\em common noun}
(like {\sl ``book''} or {\sl ``person''}); $np$ stands for {\em noun phrase}
(like {\sl ``John''} or {\sl ``the book''});
$s$ stands for the whole {\em sentence.} Actually, $n$ and $np$ represent not only
isolated nouns and noun phrases, but also syntactic groups with
similar properties: {\em e.g.,} {\sl ``the red book''} or 
{\sl ``the person whom John met yesterday,''}
from the lingustic point of view, should be treated as a noun
phrase ($np$) as well. The latter cannot be proved to be of
type $np$ by means of $\Lu$, but $\CC$ can handle this.

For simplicity, in our examples we don't distinguish singular and plural forms.

Other parts of speech receive compound types: $np \BS s$ stands
for {\em intransitive verb} (like {\sl ``runs''} or {\sl ``sleeps''}); $(np  \BS s) \SL np$
stands for {\em transitive verb} ({\sl ``likes,'' ``reads,'' ``met,'' ``admire''});
$(np \BS s) \BS (np \BS s)$ is the type for {\em adverbs} like {\sl ``yesterday''}
(it takes an intransitive verb group from the left-hand side and
yields a compound intransitive verb group); $np \SL n$  is the type
for {\sl ``the,''} etc.

If the sequent $A_1, \dots, A_n \to B$ is derivable in the Lambek calculus
or its extension, syntactic objects of type $A_1, \dots, A_n$, taken together
in the specified linear order, are considered to form an object of type $B$. For example,
since $(np \BS s) \SL np, np \SL n, n \to np \BS s$ is derivable, 
{\sl ``reads the book''} is an
expression of type $np \BS s$, or, in other words, acts as an intransitive verb.

\begin{example} \strut
\begin{center}
{\sl ``John met Pete.''} 
\qquad
{\sl ``John met Pete yesterday.''}
\end{center}
These two sentences receive type $s$, since the sequents
$np, (np \BS s) \SL np, np \to s$ and $np, (np \BS s) \SL np, np, (np \BS s) \BS (np \BS s) \to s$
are both derivable in $\Lu$.
\end{example}

\begin{example}\strut
\begin{center}
{\sl ``the person whom John met''}
\end{center}
As mentioned above, we want this phrase to receive type $np$. This is obtained
by assigning type $(n \BS n) \SL (s \SL np)$ to {\sl ``whom.''}
{\sl ``John met''} has type $s \SL np$, which means ``a sentence that lacks
a noun phrase on the right-hand side.'' In other terms, 
we have a {\em gap} after {\sl ``met.''}
In Example~1 this gap is filled by {\sl ``Pete,''} and here it is intentionally left blank.
\end{example}

\begin{example} \strut\label{Ex:yesterday}
\begin{center}
{\sl ``the person whom John met yesterday''}
\end{center}
Here the gap appears in the middle of the clause (between {\sl ``met''} and {\sl ``yesterday''}),
therefore {\sl ``John met yesterday''} is neither of type $s \SL np$, nor of type $np \BS s$.
This situation is called {\em medial extraction} and is not handled by $\Lu$.

To put $np$ into the gap, we use ${!}$ and the $(\PERM_1)$ rule: 
$$
\infer{
\begin{matrix}
np \SL n, & n, & (n \BS n) \SL (s \SL {!}np), & np, & (np \BS s) \SL np, & (np \BS s) \SL (np \BS s) & \to np\\
\text{\sl the} & \text{\sl person} & \text{\sl whom} & \text{\sl John} &  \text{\sl met} &
\text{\sl yesterday}
\end{matrix}
}
{np \SL n, n, n \BS n \to np\qquad & 
\infer{np, (np \BS s) \SL np, (np \BS s) \SL (np \BS s) \to s \SL {!} np}
{\infer[(\PERM_1)]{np, (np \BS s) \SL np, (np \BS s) \SL (np \BS s), {!}np \to s}
{\infer[({!}\to)]{np, (np \BS s) \SL np, {!}np, (np \BS s) \SL (np \BS s) \to s}
{np, (np \BS s) \SL np, np, (np \BS s) \SL (np \BS s) \to s}}}}
$$

The sequent on top is the same schema as for {\sl ``John met Pete yesterday''} (see Example~1).

Note that $s \SL {!}np$ and ${!}np \BS s$ are equivalent (due to the permutation rules).
\end{example}

\begin{example}\strut
\begin{center}
{\sl ``the paper that John signed without reading''}
\end{center}

Finally, this is the case called {\em parasitic extraction,}
with two $np$ gaps (after {\sl ``signed''} and after {\sl ``reading''}).
If the {\sl that}-clause were an independent sentence, the gaps would
have been filled like this:
{\sl ``John signed the paper without reading the paper.''} To fill both
 gaps with the same $np$, we use the $(\CONTR)$ rule:
 
$$
\small
\infer{
\begin{matrix}
np \SL n,  n, & (n \BS n) \SL (s \SL {!}np), & np, & 
(np \BS s) \SL np, & ((np \BS s) \SL (np \BS s)) \SL np, &
np \SL np & \to np\\
\text{\sl the paper} & \text{\sl that} &
\text{\sl John} & \text{\sl signed} & \text{\sl without} &
\text{\sl reading}
\end{matrix}
}{np \SL n, n, n \BS n \to np  &
\infer{np, (np \BS s) \SL np, ((np \BS s) \SL (np \BS s)) \SL np, np \SL np \to
s \SL {!}np}
{\infer{np, (np \BS s) \SL np, ((np \BS s) \SL (np \BS s)) \SL np, np \SL np, {!}np \to s}
{\infer{np, (np \BS s) \SL np, ((np \BS s) \SL (np \BS s)) \SL np, np \SL np, {!}np, {!}np \to s}
{\infer{np, (np \BS s) \SL np, {!}np, ((np \BS s) \SL (np \BS s)) \SL np, np \SL np,  {!}np \to s}
{np, (np \BS s) \SL np, np, ((np \BS s) \SL (np \BS s)) \SL np, np \SL np,  np \to s}}}}}
$$

Here {\sl ``that''} acts exactly as {\sl ``whom,''} and {\sl ``without''}
modifies the verb group like {\sl ``yesterday''} does, but also requires
a noun phrase {\sl ``reading the paper''} on the right side.
The sequents on the top
 are easily derivable in $\Lu$.

\end{example}

\begin{remark}
Our calculus $\CC$, as well as $\Dbb$, works well for pure complex sentences and pure compound
sentences. However, we meet with difficulties in the mixed case, caused
by sophisticated nature of {\sl ``and''} and the like.
For example, the fact that {\sl ``John met Pete yesterday and Mary met Ann today''}
has type $s$, leads to and unwanted classification of 
{\sl *``the person whom John met yesterday and Mary met Ann today''}
as a noun phrase (type $np$), cf. Example~\ref{Ex:yesterday}. 
In order to address this issue, Morrill
and Valent{\'\i}n~\cite{MorVal} suggest another variant of the
system, denoted by $\Dbbr$. This variant includes {\em brackets} that
disallow gapping in certain situations. Morrill and Valent{\'\i}n pose
the decidability question both for $\Dbb$ and $\Dbbr$. In this paper
we solve the first question.
\end{remark}

\begin{remark}
The whole system $\CC$ turns out to be undecidable
(Theorem~\ref{Th:main}). On the other hand, notice that in these
examples and the like can be treated using types of a very restricted form. Namely, ${!}$ is
applied only to a primitive type (for instance, ${!}np$). 
In Section~\ref{S:decidable} we show that this restricted fragment of $\CC$
is decidable. Moreover, it belongs to NP, {\em i.e.,} can be resolved by
a nondeterministic polynomial algorithm.
\end{remark}

\section{$\Lu$ with Buszkowski's Rules}

In this section we build an undecidable extension of $\Lu$
with a finite set of rules,
generally following the construction by W.~Buszkowski from~\cite{Buszko1982}.
Buszkowski, however, considers another version of the Lambek calculus,
$\mathbf{L}$, introduced in~\cite{Lambek1958}.
The difference between $\mathbf{L}$ and $\Lu$ is the so-called Lambek's
restriction: in $\mathbf{L}$, the antecedents of all sequents are forced to be non-empty.
In this paper, following Morrill and Valent\'{\i}n~\cite{MorVal}, we allow empty antecedents, and Lambek's restriction is not
valid in $\Lu$ ({\em e.g.,} $\to p\SL p$ is derivable in $\Lu$). The relationship
between $\mathbf{L}$ and $\Lu$
is very subtle. For instance, the sequent $q \SL (p \SL p) \to q$
is derivable in $\Lu$, but becomes underivable when Lambek's restiction is
imposed (despite the fact that this sequent itself has a non-empty antecedent). Therefore
one has to be very cautious with this issue, and for this reason here we provide a modification
of Buszkowski's construction for $\Lu$ rather than directly use results 
from~\cite{Buszko1982}.

Let $\Lu + \Rc$ be $\Lu$ extended
with a finite set 
$\Rc$ of rules of two special forms:
$$
\infer[(\BUSZK_1)]
{\Pi_1, \Pi_2 \to r}{\Pi_1 \to p & \Pi_2 \to q}
\qquad\raisebox{5pt}{\mbox{or}}\qquad
\infer[(\BUSZK_2),]
{\Pi \to r}{\Pi, q \to p}
$$
where $p, q, r$ are fixed {\em primitive} types.
We call these rules {\em Buszkowski's rules.}

\begin{theorem}\label{Th:Buszkcutelim}
The cut rule 
$$
\infer[(\CUT)]
{\Delta_1, \Pi, \Delta_2 \to C}
{\Pi \to A & \Delta_1, A, \Delta_2 \to C}
$$
is admissible in $\Lu + \Rc$ for an arbitrary
set $\Rc$ of Buszkowski's rules.
\end{theorem}

\begin{proof}
We proceed by double induction. 
We consider a number of cases, and in each of them the cut either
disappears, or is replaced by cuts with simpler cut formulae ($A$),
or is replaced by a cut for which the depth of at least one
derivation tree of a premise ($\Pi \to A$ or $\Delta_1, B, \Delta_2 \to C$) is
less than for the original cut, and the cut formula remains the same.
Thus by double induction (on the outer level---on the complexity of $A$,
on the inner level---on the sum of premise derivation tree depths) we get
rid of the cut.

{\bf Case 1:} $A$ is not the type that is introduced by the lowermost rule
in the derivation of $\Delta_1, A, \Delta_2 \to C$.
In this case $(\CUT)$ can be interchanged with that lowermost rule.
Consider the situation when it was $(\BUSZK_1)$ (other cases are similar):
$$
\small
\infer[(\CUT)]
{\Delta_1', \Pi, \Delta_1'', \Delta_2 \to r}
{\Pi \to A & \infer[(\BUSZK_1)]{\Delta_1', A, \Delta_1'', \Delta_2 \to r}
{\Delta_1', A, \Delta_1'' \to p & \Delta_2 \to q}}
$$
\centerline{\turnbox{270}{$\leadsto$}}
\vskip 10pt
$$
\small
\infer[(\BUSZK_1)]
{\Delta_1', \Pi, \Delta_1'', \Delta_2 \to r}
{\infer[(\CUT)]{\Delta_1', \Pi, \Delta_1'' \to p}
{\Pi \to A & \Delta_1', A, \Delta_1'' \to p}
 & \Delta_2 \to q}
$$

{\bf Case 2:} $A = E \SL F$, and it is introduced by the lowermost rules both
into $\Pi \to A$ and into $\Delta_1, A, \Delta_2 \to C$.

$$
\small
\infer[(\CUT)]
{\Delta_1, \Gamma, \Pi, \Delta_2 \to C}
{\infer[(\to\SL)]{\Gamma \to E \SL F}{\Gamma, F \to E}
 &
\infer[(\SL\to)]{\Delta_1, E \SL F, \Pi, \Delta_2 \to C}
{\Pi \to F & \Delta_1, E, \Delta_2 \to C}}
$$
\centerline{\turnbox{270}{$\leadsto$}}
\vskip 10pt
$$
\small
\infer[(\CUT)]
{\Delta_1, \Gamma, \Pi, \Delta_2 \to C}
{\infer[(\CUT)]{\Gamma, \Pi \to E}{\Pi \to F & \Gamma, F \to E}
& \Delta_1, E, \Delta_2 \to C}
$$

Case 2 for $\BS$ is handled symmetrically.

{\bf Case 3:} one of the premises of $(\CUT)$ is the axiom
($A \to A$). Then the goal coincides with the other premise.

Note that since $(\BUSZK_1)$ and $(\BUSZK_2)$ introduce
new primitive types only into the succedent, the ``bad'' case,
where both premises of the cut rule are derived using Buszkowski's
rules and the cut formula is the formula introduced by both of them,
does not occur. This is the key trick that allows to formulate
the extended calculus in a cut-free way. 
\end{proof}

In the presence of $(\CUT)$ Buszkowski's
rules $(\BUSZK_1)$ and $(\BUSZK_2)$ are equivalent
to axioms $p,q\to r$ and $p \SL q \to r$ respectively,
as shown by the following derivations:

$$
\infer[(\CUT)]
{\Pi_1, \Pi_2 \to r}
{\Pi_1 \to p & 
\infer[(\CUT)]{p, \Pi_2 \to r}
{\Pi_2 \to q & p, q \to r}}
\qquad
\infer[(\CUT)]
{\Pi \to r}
{\infer[(\to\SL)]{\Pi \to p \SL q}
{\Pi, q \to p} & p \SL q \to r}
$$
and in the opposite direction:
$$
\raisebox{10pt}{
\infer[(\BUSZK_1)]{\strut p, q \to r}{p \to p & q \to q}
}
\qquad
\infer[(\BUSZK_2)]{p \SL q \to r}{
\infer[(\SL\to)]{p \SL q, q \to p}{q \to q & p \to p}}
$$

From this perspective, $\Lu + \Rc$ can be viewed as
a finite {\em axiomatic extension} of $\Lu$ (with 
non-logical axioms of a special kind).
However, for our purposes it is more
convenient to consider rules instead of axioms.

\begin{theorem}\label{Th:Buszkrecursive}
Let $M$ be a recursively enumerable set of words over an alphabet $\Sigma$
without the empty word. If
 $\Sigma \subset \Var$, and $\Var$ also
 contains an infinite number of variables not
 belonging to $\Sigma$, then there exists a finite
 set $\Rc_M$ of Buszkowski's rules and $s \in \Var$ 
 such that for any word $a_1 \dots a_n$ over $\Sigma$
$$
a_1 \dots a_n \in M \text{\quad if{f}\quad }
a_1, \dots, a_n \to s \text{ is derivable in $\Lu + \Rc_M$}.
$$
\end{theorem}

We shall use the fact that any
 recursively enumerable language without the empty
word can be generated by a 
{\em binary grammar}~\cite{Buszko2005}. A binary
grammar is a quadruple $G = \langle N, \Sigma, P, s 
\rangle$, where $N$ and $\Sigma$ are disjoint alphabets
($\Sigma$ is the original alphabet of the language),
$s \in N$, and $P$ is a finite set of 
{\em productions} of the form%
\footnote{In the definition from~\cite{Buszko2005},
$P$ could also include productions
 of the form $u \Rightarrow v$
for $u, v \in N \cup \Sigma$. Such a rule can be
equivalently replaced by two productions 
$u \Rightarrow w_1 w_2$, $w_1 w_2 \Rightarrow v$,
where $w_1$ and $w_2$ are new elements added to $N$ 
(different for different rules). We encode these simple
productions using more complex ones in order to reduce
the number of cases to be considered in the proofs.}
$$w \Rightarrow v_1 v_2 \text{\quad or\quad } v_1 v_2 \Rightarrow w,$$ where
$v_1, v_2, w \in N \cup \Sigma$.
If $(\alpha \Rightarrow \beta) \in P$ and $\eta$ and
$\theta$ are arbitrary (possibly empty) words over $N \cup \Sigma$, then
$\eta\alpha\theta \Rightarrow_G \eta\beta\theta$. The
relation $\Rightarrow^*_G$ is the reflexive-transitive
closure of $\Rightarrow_G$. Finally, the language
generated by $G$ is the set of all words $a_1 \dots a_n$
over $\Sigma$ such that $s \Rightarrow^*_G a_1 \dots a_n$.

\begin{proof}
Let $M$ be an arbitrary recursively enumerable language without
the empty word and
let $G$ be a binary grammar that generates $M$. 
We construct the corresponding
extension of $\Lu$. Let $N \cup \Sigma \subset \Var$,
and let $\Var$ contain an infinite number of extra
fresh variables that we'll need later.
For every production 
$(w \Rightarrow v_1 v_2) \in P$ we add one rule
$$
\infer[(\EASY)]
{\Delta_1, \Delta_2 \to w}
{\Delta_1 \to v_1 & \Delta_2 \to v_2}
$$

For productions of the form $v_1 v_2 \Rightarrow w$ the
 construction
is more complex. First for every pair
$\PP = \langle (v_1 v_2 \Rightarrow w), 
x \rangle$, where 
$(v_1 v_2 \Rightarrow w) \in P$ and 
$x \in N \cup \Sigma$,
we introduce new variables
$\TQ$ for every $y \in N \cup \Sigma$ and
five extra variables $\Ba$, $\Bb$, $\Bc$, $\Be$,
$\Bf$. Then  for every $\PP$ we add the following rules.
Some of these rules are not in Buszkowski's form. 
We'll transform them into the correct format below.

\begin{center}
\begin{tabular}{c@{\qquad\qquad}c}
$
\infer[(1_\PP)]
{\Delta_1, \Delta_2 \to \Ba}
{\Delta_1 \to \Be & \Delta_2 \to x}
$
&
$
\infer[(2_\PP)]
{\Delta_1, \Delta_2 \to \Ba}
{\Delta_1 \to \TQ & \Delta_2, y \to \Ba}
$\\[7pt]
$
\infer[(3_\PP)]
{\Delta_1, \Delta_2 \to \Bb}
{\Delta_1 \to \TW & \Delta_2, v_1, v_2 
\to \Ba}
$
&
$
\infer[(4_\PP)]
{\Delta_1, \Delta_2 \to \Bb}
{\Delta_1 \to \TQ & \Delta_2, y \to \Bb}
$
\\[7pt]
$
\infer[(5_\PP)]
{\Delta_1, \Delta_2 \to \Bc}
{\Delta_1 \to \Bf & \Delta_2, \Be \to \Bb}
$
&
$
\infer[(6_\PP)]
{\Delta_1, \Delta_2 \to \Bc}
{\Delta_1 \to y  & \Delta_2, \TQ \to \Bc}
$
\\[7pt]
$
\infer[(7_\PP)]
{\Delta \to x}
{\Delta, \Bf \to \Bc}
$
\end{tabular}
\end{center}

As already said, some of these rules are not actually Buszkowski's
rules. However, any rule of the form
$$
\infer{\Delta_1, \Delta_2 \to t}
{\Delta_1 \to p & \Delta_2, q \to r}
$$
(these are rules $(2_\PP)$, $(4_\PP)$, $(5_\PP)$,
and $(6_\PP)$)
can be equivalently replaced by two rules
$$
\infer[(\BUSZK_2)]{\Delta \to u}{\Delta, q \to r}
\qquad\raisebox{5pt}{\mbox{and}}\qquad
\infer[(\BUSZK_1)]{\Delta_1, \Delta_2 \to t}
{\Delta_1 \to p & \Delta_2 \to u}$$
where $u$ is a fresh variable. 

Similarly, $(3_\PP)$ is a shortcut for three rules:
$$
\infer[(\BUSZK_2)]
{\Delta \to u_1}
{\Delta, v_2 \to \Ba}
\qquad
\infer[(\BUSZK_2)]
{\Delta \to u_2}
{\Delta, v_1 \to u_1} 
\qquad
\infer[(\BUSZK_1)]
{\Delta_1, \Delta_2 \to \Bb}
{\Delta_1 \to \TW & \Delta_2 \to u_2}
$$

Rules $(1_\PP)$, $(7_\PP)$, and $(\EASY)$ are already
in the correct format.  Thus we've actually constructed a calculus of the
form $\Lu + \Rc$. Denote it by $\Lu + \Rc_M$.

Now to achieve our goal it is sufficient to prove
that for any $x, z_1, \dots, z_m \in N \cup \Sigma$
$$
x \Rightarrow^*_G z_1 \dots z_m 
\text{\qquad if{f} \qquad}
z_1, \dots, z_m \to x
\text{ is derivable in $\Lu + \Rc_M$.}
$$

The proof consists of two directions.

\fbox{$\Leftarrow$}
All types in the sequent $z_1, \dots, z_m \to x$ 
are primitive, therefore its derivation includes only
axioms and
Buszkowski's rules, but not original rules of $\Lu$
($(\to\SL)$, $(\to\BS)$, $(\SL\to)$, $(\BS\to)$).

Since $\Be$, $\Bf$, and $\TQ$ (for all $y \in N
\cup \Sigma$, including $w$) do not appear
in the succedents of goal sequents in 
Buszkowski's rules from $\Rc_M$,
the only possible situation when 
$\Be$, $\Bf$, or $\TQ$ actually 
appears in the succedent is the axiom. Hence rules
$(1_\PP)$--$(5_\PP)$ can be rewritten in a simpler way
(rules $(6_\PP)$ and $(7_\PP)$ are not affected by
this simplification):

\begin{center}
\begin{tabular}{c@{\qquad}c@{\qquad}c@{\qquad}}
$
\infer[(1'_\PP)]
{\Be, \Phi \to \Ba}{\Phi \to x}
$
&
$
\infer[(2'_\PP)]
{\TQ, \Phi \to \Ba}{\Phi, y \to \Ba}
$
&
$
\infer[(3'_\PP)]
{\TW, \Phi \to \Bb}{\Phi, v_1, v_2 \to \Ba}
$
\\[7pt]
$
\infer[(4'_\PP)]
{\TQ, \Phi \to \Bb}{\Phi, y \to \Bb}
$
&
$
\infer[(5'_\PP)]
{\Bf, \Phi \to \Bc}{\Phi, \Be \to \Bb}
$
&
$
\infer[(6_\PP)]
{\Delta_1, \Delta_2 \to \Bc}
{\Delta_1 \to y  & \Delta_2, \TQ \to \Bc}
$
\\[7pt]
$
\infer[(7_\PP)]
{\Phi \to x}
{\Phi, \Bf \to \Bc}
$
\end{tabular}
\end{center}

Proceed by induction on the cut-free derivation.
The sequent $z_1, \dots, z_m \to x$
 could either be an axiom
(and then $n = 1$, $z_1 = x$, and trivially 
$x \Rightarrow^*_G x$) or be derived by one of
the Buszkowski's rules. Since $x \in N \cup \Sigma$,
the only possible rules are $(\EASY)$ and $(7_\PP)$.

If $z_1, \dots, z_m \to x$ is derived using 
$(\EASY)$:
$$
\infer[(\EASY),]
{\strut z_1, \dots, z_k, z_{k+1}, \dots z_m \to x}
{z_1, \dots, z_k \to v_1 & z_{k+1}, \dots, z_m \to v_2}
$$
 then we have $z_1, \dots, z_k \to v_1$,
$z_{k+1}, \dots, z_m \to v_2$, and 
$(x \Rightarrow v_1 v_2) \in P$. By induction
hypothesis, $v_1 \Rightarrow^*_G z_1 \dots z_k$
and $v_2 \Rightarrow^*_G z_k \dots z_n$, therefore
we get $x \Rightarrow_G v_1 v_2 \Rightarrow^*_G
z_1 \dots z_k z_{k+1} \dots z_m$.

If the last rule in the derivation is $(7_\PP)$, 
then we get $z_1, \dots, z_m, \Bf \to \Bc$. Trace
the type in the succedent. Since the antecedent doesn't
contain $\Bc$, $\Bb$, or $\Ba$, sequents with these types
in the succedent could not appear as axioms, and
the only ways to derive such sequents are represented
by the following schema (the arrows go from goal
to premises):

$$
\xymatrix{
\ar[r] &
\ar@(ul,ur)^{(6_\PP)} \Bc \ar[r]_{(5'_\PP)} &
\ar@(ul,ur)^{(4'_\PP)} \Bb \ar[r]_{(3'_\PP)} &
\ar@(ul,ur)^{(2'_\PP)} \Ba \ar[r]_{(1'_\PP)} &
x
}
$$

Therefore, the sequent 
$z_1, \dots, z_m, \Bf \to \Bc$ is derived 
in the following way: several 
(possibly zero) applications of
$(6_\PP)$, then $(5'_\PP)$, then several
$(4'_\PP)$, then several $(2'_\PP)$, then $(1'_\PP)$.
Finally, on top of this last $(1'_\PP)$ rule we
again get a sequent with $x$ in the succedent.
The whole derivation has the following form.
Here $^*$ means several consecutive applications of
the same rule, and
$\Delta_1, \dots, \Delta_n = z_1, \dots, z_m$.

{\small
$$
\infer[(7_\PP)]
{\Delta_1, \dots, \Delta_n \to x}
{
\infer[(6_\PP)^*]{
\Delta_1, \dots, \Delta_n, \Bf \to \Bc}
{\Delta_1 \to y_1 & \dots 
& \Delta_k \to w 
 & \dots
& 
\Delta_n \to y_n &
\infer[(5'_\PP)]{\Bf, \TQ_1, \dots, \TQ_{k-1}, \TW, \TQ_{k+1}, \dots, \TQ_n \to \Bc}{
\infer[(4'_\PP)^*]{\TQ_1, \dots, \TQ_{k-1}, \TW, \TQ_{k+1}, \dots, \TQ_n, \Be \to \Bb}{
\infer[(3'_\PP)]{\TW, \TQ_{k+1}, \dots, \TQ_n, \Be,
y_1, \dots, y_{k-1} \to \Bb}
{
\infer[(2'_\PP)^*]{\TQ_{k+1}, \dots, \TQ_n, \Be, y_1, \dots, y_{k-1}, v_1, v_2 \to \Ba}
{
\infer[(1'_\PP)]{\Be, y_1, \dots, y_{k-1}, v_1, v_2, 
y_{k+1}, \dots, y_n \to \Ba}
{y_1, \dots, y_{k-1}, v_1, v_2, y_{k+1}, \dots, y_n \to x}
}}}}}}
$$
}

Here rule $(1'_\PP)$ introduces $\Be$, $(2'_\PP)$ moves $\Be$
to the left and marks $y_i$ as $\TQ_i$, $(3'_\PP)$ actually applies
the production $(v_1 v_2 \Rightarrow w)$, which is possible, since
$v_1, v_2$ is now on the edge of the antecedent, $(4'_\PP)$ continues
the movement, and finally $(5'_\PP)$, $(6_\PP)$, and $(7_\PP)$ move the
letters backwards, unmark them and return the antecedent to $x$.
 
By induction hypothesis, $y_1 \Rightarrow^*_G \Delta_1$,
\dots, $y_{k-1} \Rightarrow^*_G \Delta_{k-1}$,
$w \Rightarrow^*_G \Delta_k$,
$y_{k+1} \Rightarrow^*_G \Delta_{k+1}$,
\dots,
$y_n \Rightarrow^*_G \Delta_n$, and
$x \Rightarrow^*_G y_1 \dots y_{k-1} v_1 v_2 y_{k+1}
\dots y_n$. By application of 
$v_1 v_2 \Rightarrow w$ we get
$x \Rightarrow^*_G \Delta_1 \dots \Delta_n$.

We notice that the first type of productions of the binary grammar
is handled much easier than the second one. This is due to the fact
that in the first case we simulate standard context-free derivation,
while in the second case the production is not context-free and even
not context-sensitive.

\fbox{$\Rightarrow$}
Proceed by induction on $\Rightarrow^*_G$. For
the base case ($x \Rightarrow^*_G x$) the corresponding
sequent ($x \to x$) is an axiom. 

If the last production
is $w \Rightarrow v_1 v_2$:
$$x \Rightarrow^*_G z_1 \dots z_{k-1} w z_{k+1} \dots
z_m \Rightarrow_G z_1 \dots z_{k-1} v_1 v_2 z_{k+1} \dots
z_m,$$ then by induction hypothesis
$z_1, \dots, z_{k-1},  w, z_{k+1}, \dots, z_m \to x$ 
is derivable in $\Lu + \Rc_M$. Also
$v_1, v_2 \to w$ is derivable by $(\BUSZK_1)$, and
by $(\CUT)$ we obtain $z_1, \dots, z_{k-1},\linebreak v_1, v_2, z_{k+1},
\dots, z_m \to x$. The cut rule is admissible
in $\Lu + \Rc$ by Theorem~\ref{Th:Buszkcutelim}.

For the $v_1 v_2 \Rightarrow w$ case, {\em i.e.,} the last
production is applied like this:
$$x \Rightarrow^*_G z_1 \dots
z_{k-1} v_1 v_2 z_{k+1} \dots z_m \Rightarrow_G
z_1 \dots z_{k-1} w z_{k+1} \dots z_m,$$ the derivation is
as follows (here $\PP = \langle
(v_1 v_2 \Rightarrow w), x \rangle$):

$$
\small
\infer[(7_\PP)]
{\strut z_1, \dots, z_{k-1}, w, z_{k+1}, \dots, z_m \to x}
{\infer[(6_\PP)^*]
{z_1, \dots, z_{k-1}, w, z_{k+1}, \dots, z_m, \Bf \to \Bc}
{z_1 \to z_1 & \dots & w \to w & \dots & z_m \to z_m &
\infer[(5'_\PP)]{
\Bf, \TZ_1, \dots, \TZ_{k-1}, \TW, \TZ_{k+1}, \dots,
\TZ_m \to \Bc}
{
\infer[(4'_\PP)^*]
{\TZ_1, \dots, \TZ_{k-1}, \TW, \TZ_{k+1}, \dots,
\TZ_m, \Be \to \Bb}
{
\infer[(3'_\PP)]
{\TW, \TZ_{k+1}, \dots, \TZ_m, \Be, z_1, \dots, z_{k-1}
\to \Bb}
{
\infer[(2'_\PP)^*]
{\TZ_{k+1}, \dots, \TZ_m, \Be, z_1, \dots, z_{k-1},
v_1, v_2 \to \Ba}
{
\infer[(1'_\PP)]
{\Be, z_1, \dots, z_{k-1}, v_1, v_2, z_{k+1}, \dots,
z_m \to \Ba}
{z_1, \dots, z_{k-1}, v_1, v_2, z_{k+1}, \dots, z_m 
\to x}
}}}}}}
$$

The sequent on the top is derivable by
inductive hypothesis. 
\end{proof}

Since there exist undecidable recursively
enumerable languages, Theorem~\ref{Th:Buszkrecursive} now yields the
following result:

\begin{theorem}\label{Th:Buszkundec}
There exists a finite set of Buszkowski's rules $\Rc_0$ such that the derivability
problem for $\Lu + \Rc_0$ is undecidable.
\end{theorem}

Note that the $\BS$ connective is not used in the
construction, so we've actually obtained undecidability
for $\Luu + \Rc_0$.

\section{Undecidability of $\CC$}\label{S:undec}

We prove undecidability of $\CC$ 
by encoding $\Lu+\Rc$ derivations in this calculus. In order
to do that, we first prove a technical proposition.

If $\Rc$ is a set of Buszkowski's rules, let
$$
\Gc_\Rc = \left\{ (r \SL q) \SL p \mid
\tfrac{\Pi_1 \to p \quad \Pi_2 \to q}{\Pi_1, \Pi_2 \to r} 
\in \Rc \right\} \cup
\left\{ r \SL (p \SL q) \mid 
\tfrac{\Pi, q \to p}{\Pi \to r} 
\in \Rc \right\}.
$$

If $\Bx = \{ B_1, \dots, B_n \}$ is a finite set of 
formulae, let $\GAc{\Bx} = {!}B_1, \dots, {!}B_n$. (The order
of the elements in $\Bx$ doesn't matter, since $\GAc{\Bx}$
will appear in left-hand sides of $\CC$ sequents, and
in $\CC$ we have the $(\PERM_{1,2})$ rules.)

\begin{theorem}\label{Th:deduction}
$\Lu + \Rc \vdash \Pi \to A$ if and only if there exists
$\Bx \subseteq \Gc_\Rc$ such that $\CC \vdash 
\GAc{\Bx}, \Pi \to A$.
\end{theorem}

In this theorem a finite {\em theory} ($\Rc$) that extends the basic calculus
($\Lu$) gets embedded into the formula (more precisely, the sequent $\Pi \to A$)
being derived. In linear logic this is possible with the help of
the exponential modality (${!}$). However, our version of ${!}$ doesn't 
enjoy the weakening rule, therefore we cannot
always take $\Bx = \Gc_\Rc$, as one usually could expect.
Generally, with $\Bx = \Gc_\Rc$ the ``only if'' statement
is false. For example, ${!}(r \SL (p \SL q)), s \to s$ is not
derivable in $\CC$, but $s \to s$ is indeed derivable in $\Lu + \Rc$
for any $\Rc$. This happens because this particular Buszkowski's rule,
encoded by $r \SL (p \SL q)$, is not {\em relevant} to $s \to s$.

\begin{proof}

\fbox{$\Rightarrow$}
Proceed by induction.

 If $\Pi \to A$ is an axiom ($A \to A$), just take $\Bx = \varnothing$.

If $A = B \SL C$, and $\Pi \to A$ is obtained using
the $(\to\SL)$ rule from
$\Pi, C \to B$, then take the same $\Bx$ and apply
the same rule:
$$
\infer{\GAc{\Bx}, \Pi \to B \SL C}
{\GAc{\Bx}, \Pi, C \to B}
$$

If $\Pi = \Phi_1, B \SL C, \Psi, \Phi_2$, and
$\Pi \to A$ is obtained by $(\SL\to)$ from $\Psi \to C$
and $\Phi_1, B, \Phi_2 \to A$, then by induction
hypothesis $\CC \vdash \GAc{\Bx_1}, \Psi \to C$
and $\CC \vdash \GAc{\Bx_2}, \Phi_1, B, \Phi_2
\to A$ for some $\Bx_1, \Bx_2 \subseteq \Gc_\Ac$. Let
$\Bx = \Bx_1 \cup \Bx_2$. Then for
$\GAc{\Bx}, \Pi \to A$ we have the following
derivation in $\CC$, where ${}^*$ means several applications
of the rules in any order.
$$
\infer[(\CONTR,\PERM)^*]
{\GAc{\Bx_1 \cup \Bx_2},
\Phi_1, B \SL C, \Psi, \Phi_2 \to A}
{
\infer[(\SL\to)]{\GAc{\Bx_2}, \Phi_1, B \SL C, 
\GAc{\Bx_1}, \Psi, \Phi_2 \to A}
{\GAc{\Bx_1}, \Psi \to C &
\GAc{\Bx_2}, \Phi_1, B, \Phi_2 \to A}
}
$$

Finally, $\Pi \to A$ can be obtained by application of
Buszkowski's rules $(\BUSZK_1)$ or $(\BUSZK_2)$. In the
first case, $A = r$, $\Pi = \Pi_1, \Pi_2$,
and both $\Pi_1 \to p$ and $\Pi_2 \to q$ are
derivable in $\Lu + \Rc$. Thus by induction
hypothesis we get $\CC \vdash \GAc{\Bx_1},
\Pi_1 \to p$ and $\CC \vdash \GAc{\Bx_2},
\Pi_2 \to q$ for some $\Bx_1, \Bx_2 \subseteq \Gc_\Rc$.
Moreover, $(r \SL q) \SL p \in \Gc_\Rc$. Now take
$\Bx = \Bx_1 \cup \Bx_2 \cup \{ (r \SL q) \SL p \}$
and enjoy the following derivation for
$\GAc{\Bx}, \Pi_1, \Pi_2 \to r$ in $\CC$:

$$
\infer[(\CONTR,\PERM)^*]
{\GAc{\Bx}, \Pi_1, \Pi_2 \to r}
{
\infer[({!}\to)]{
{!}((r \SL q) \SL p), \GAc{\Bx_1},
\Pi_1, \GAc{\Bx_2}, \Pi_2 \to r
}
{
\infer[(\SL\to)]
{(r \SL q) \SL p, \GAc{\Bx_1},
\Pi_1, \GAc{\Bx_2}, \Pi_2 \to r}
{\GAc{\Bx_1}, \Pi_1 \to p &
\infer[(\SL\to)]
{r \SL q, \GAc{\Bx_2}, \Pi_2 \to r}
{\GAc{\Bx_2}, \Pi_2 \to q & r \to r}}}
}
$$

In the $(\BUSZK_2)$ case, $A = r$, and we have
$\GAc{\Bx'}, \Pi, q \to p$ in the
induction hypothesis for some $\Bx' \subseteq
\Gc_\Rc$. Let $\Bx = \Bx' \cup \{ r \SL (p \SL q) \}$
(recall that $r \SL (p \SL q) \in \Gc_\Rc$), and proceed
like this:
$$
\infer[(\CONTR,\PERM)^*]{\GAc{\Bx}, \Pi \to r}
{
\infer[({!}\to)]
{{!}(r \SL (p \SL q)), \GAc{\Bx'}, \Pi \to r}
{\infer[(\SL\to)]{r \SL (p \SL q), \GAc{\Bx'}, \Pi \to r}
{
\infer[(\to\SL)]
{\GAc{\Bx'}, \Pi \to p \SL q}
{\GAc{\Bx'}, \Pi, q \to p}
& r \to r}
}
}
$$

\fbox{$\Leftarrow$}
Recall that if $\Bx = \{ B_1, \dots, B_n \}$ is a finite
set of formulae, then $\GAc{\Bx} = {!}B_1, \dots, {!}B_n$
(as stated above, the
order of the elements in $\Bx$ doesn't matter due to the $(\PERM_{1,2})$ rules).
For deriving sequents of the form $\GAc{\Bx}, \Pi \to C$,
where $\Pi$, $C$, and $\Bx$ do not contain ${!}$ and $\BS$,
 one can use a simpler calculus than
$\CC$:
$$
\infer{p \to p}{}
\qquad
\infer[(\to\SL)]{\GAc{\Bx}, \Pi \to A \SL B}
{\GAc{\Bx}, \Pi, B \to A}
$$
$$
\infer[(\SL\to)]{
\GAc{\Bx_1 \cup \Bx_2}, \Delta_1, A \SL B, 
\Pi, \Delta_2 \to C}
{\GAc{\Bx_1}, \Pi \to B &
\GAc{\Bx_2}, \Delta_1, A, \Delta_2 \to C}
\qquad
\infer[({!}\to)]{\GAc{\Bx \cup \{ A \} },
\Delta_1, \Delta_2 \to C}{\GAc{\Bx},
\Delta_1, A, \Delta_2 \to C}
$$

Moreover, the $({!}\to)$ rule is interchangeable
with the others in the following ways:

$$
\infer[({!}\to)]{\GAc{\Bx \cup \{C\}}, \Delta_1, \Delta_2 \to A \SL B}
{\infer[(\to\SL)]
{\GAc{\Bx}, \Delta_1, C, \Delta_2 \to A \SL B}
{\GAc{\Bx}, \Delta_1, C, \Delta_2, B \to A}
}
\qquad
\text{\raisebox{1.2em}{$\leadsto$}}
\qquad
\infer[(\to\SL)]{\GAc{\Bx \cup \{C\}},
 \Delta_1, \Delta_2 \to A \SL B}
{\infer[({!}\to)]
{\GAc{\Bx \cup \{C\}}, \Delta_1, \Delta_2, B \to A}
{\GAc{\Bx}, \Delta_1, C, \Delta_2, B \to A}
}
$$

\vskip 10pt

$$
\infer[({!}\to)]{\GAc{\Bx_1 \cup \Bx_2 \cup \{D\}}, \Delta_1, A \SL B, \Pi, \Delta'_2, \Delta''_2 \to C}
{\infer[(\SL\to)]
{\GAc{\Bx_1 \cup \Bx_2}, \Delta_1, A \SL B, \Pi, \Delta'_2, D, \Delta''_2 \to C}
{\GAc{\Bx_1}, \Pi \to B & \GAc{\Bx_2}, \Delta_1, A, \Delta'_2, D, \Delta''_2 \to C}
}
$$
\centerline{\turnbox{270}{$\leadsto$}}
\vskip 10pt
$$
\infer[(\SL\to)]{\GAc{\Bx_1 \cup \Bx_2 \cup \{D\}}, \Delta_1, A \SL B, \Pi, \Delta'_2, \Delta''_2 \to C}
{\GAc{\Bx_1}, \Pi \to B & \infer[({!}\to)]
{\GAc{\Bx_2 \cup \{D\}}, \Delta_1, A, \Delta'_2, \Delta''_2 \to C}
{\GAc{\Bx_2}, \Delta_1, A, \Delta'_2, D, \Delta''_2 \to C}
}
$$
And the same, if $D$ appears inside $\Delta_1$ or $\Pi$. Finally, consecutive
applications of $({!}\to)$ are always interchangeable.

After applying these transformations, we achieve a derivation where
$({!}\to)$ is applied immediately after applying $(\SL\to)$ with the same
active type (the other case, when it is applied after the axiom to $p$, is
impossible, since $\Bx$ is always a subset of $\Gc_\Rc$, and
the latter doesn't contain sole variables).
In other words, applications of $({!}\to)$ appear only in the following
two situations:
$$
\infer[({!}\to)]{\GAc{\Bx_1 \cup \Bx_2 \cup \{(r \SL q) \SL p)\}},
\Delta_1, \Pi, \Delta_2 \to A}
{
\infer[(\SL\to)]{\GAc{\Bx_1 \cup \Bx_2}, \Delta_1, (r \SL q) \SL p, \Pi, \Delta_2 \to A}
{\GAc{\Bx_1}, \Pi \to p & \GAc{\Bx_2}, \Delta_1, r \SL q, \Delta_2 \to A}
}
$$
and
$$
\infer[({!}\to)]{\GAc{\Bx_1 \cup \Bx_2 \cup \{ r \SL (p \SL q)\}}, \Delta_1, \Pi, \Delta_2 \to A}
{\infer[(\SL\to)]{\GAc{\Bx_1 \cup \Bx_2}, \Delta_1, r \SL (p \SL q), \Pi, \Delta_2 \to A}
{\GAc{\Bx_1}, \Pi \to p \SL q & \GAc{\Bx_2}, \Delta_1, r, \Delta_2 \to A}
}
$$

Now we prove the statement
$\CC \vdash \GAc{\Bx}, \Pi \to A$ 
(where $\Bx \subseteq \Gc_\Rc$) $\Rightarrow$
$\Lu + \Rc \vdash \Pi \to A$ by induction on the
above canonical derivation. For the case of axiom
or applications of rules $(\to\SL)$ and $(\SL\to)$ we
just apply the same rules in $\Lu + \Rc$, so the only
interesting case is $({!}\to)$. Consider the two
possible situations.

In the $(r \SL q) \SL p$ case, by induction hypothesis we get 
$\Lu + \Rc \vdash \Pi \to p$ and $\Lu + \Rc \vdash \Delta_1, r \SL q, \Delta_2 \to A$,
and then we develop the following derivation
in $\Lu+\Rc$ (recall that $(\CUT)$ is admissible there):
$$
\infer[(\CUT)]{\strut\Delta_1, \Pi, \Delta_2 \to A}
{\Pi \to p & 
\infer[(\CUT)]{\strut \Delta_1, p, \Delta_2 \to A}
{\strut \infer[(\to\SL)]{p \to r \SL q}{p,q\to r} &
\Delta_1, r \SL q, \Delta_2 \to A}
}
$$

In the case of $r \SL (p \SL q)$, the derivation looks like this:
$$
\infer[(\CUT)]{\strut \Delta_1, \Pi, \Delta_2 \to A}
{\Pi \to p \SL q & 
\infer[(\CUT)]{\strut \Delta_1, p \SL q, \Delta_2 \to A}
{
p \SL q \to r
& \Delta_1, r, \Delta_2 \to A}}
$$

This completes the proof of Theorem~\ref{Th:deduction}. 
\end{proof}

Now we can return to our main claim.

\begin{proof}[Proof of Theorem~\ref{Th:main}]
Take $\Rc_0$ from Theorem~\ref{Th:Buszkundec} and
suppose that $\CC$ is decidable.
Then we can present an
algorithm that solves the derivability problem for
$\Lu + \Rc_0$. Namely, for a sequent $\Pi \to A$
we search through all subsets
$\Bx \subseteq \Gc_\Rc$ (and there is a finite number
of them) and test derivability of $\GAc{\Bx}, \Pi \to A$
in $\CC$. By Theorem~\ref{Th:deduction}, $\Pi \to A$ is
derivable in $\Lu + \Rc_0$ if and only if at least one
of these tests succeeds. This contradicts Theorem~\ref{Th:Buszkundec}.
Therefore $\CC$ is undecidable.

Since we never used $\BS$ in the construction,
we get undecidability for $\CCu$. 
\end{proof}

This proof of Theorem~1 is in the spirit of our previous work~\cite{KanKuzSce}.
The significant difference between this paper and~\cite{KanKuzSce}
is that here the modality does not satisfy the weakening rule and the system
$\CC$ doesn't obey any version of Lambek's restriction ({\em i.e.,} the antecedents
are allowed to be empty).
Due to the lack of the weakening rule, in Theorem~\ref{Th:deduction} 
it is not sufficient to check derivability only for $\Bx = \Gc_\Rc$,
and therefore Theorem~\ref{Th:deduction} is formulated in the relevant
logic style.
We also had to open up and reassemble Buszkowski's proof from~\cite{Buszko1982} and~\cite{Buszko2005}
to make it work without Lambek's restriction (in $\Lu$). 

\section{A Decidable Fragment of $\CC$}\label{S:decidable}

Undecidability of $\CC$ is somewhat unfortunate, because this
calculus is liguistically motivated (see Section~\ref{S:ling}).
However, in our examples ${!}$ was applied only to primitive
types ($np$). If we consider only sequents with this restriction,
the situation is different: the derivability problem becomes
decidable. Moreover, it belongs to NP.

Let's call the {\em size} of a formula $A$ (denoted by $|A|$) the total number of variable and
connective occurrences in $A$. More formally, $|A|$ is defined recursively:
$|p| = 1$ for $p \in \Var$, $|A \BS B| = |B \SL A| = |A| + |B| + 1$,
$|{!}A| = |A| + 1$. The size of a sequent $A_1, \dots, A_n \to B$ is
$|A_1| + \ldots + |A_n| + |B|$.

In the pure Lambek calculus, the size of any derivation
is necessarily bounded by the size of the goal sequent.
In our case, a sequent could have derivations of arbitrary
size due to uncontrolled application of permutation rules:
two consecutive applications of $(\PERM_1)$ and $(\PERM_2)$ (with
the same formula at the same places) do nothing with the sequent, but increase
the derivation size. Nevertheless, the following lemma shows that
every sequent has a derivation of quadratic size.

\begin{lemma}
If the sequent $\Pi \to C$ is derivable in $\CC$ and 
${!}$ in this sequent is applied only to variables,
then this
sequent has a derivation of size less than $12n^2 + 3n$, where
$n$ is the size of $\Pi \to C$.
\end{lemma}

Recall that $(\CUT)$ is not included in the system,
all derivations are cut-free.

\begin{proof}
We represent the derivation of $\Pi \to C$ as a tree. The leaves of the
tree are instances of axioms, and the inner nodes correspond to applications
of rules. Rules $(\SL\to)$ and $(\BS\to)$ form {\em branching points} of the tree.
The number of leaves is equal to the number of branching points plus one.

Let's call $(\PERM_{1,2})$ and $(\CONTR)$ {\em structural} rules;
other rules are {\em logical} ones.

Each logical rule introduces exactly
one connective into the goal sequent $\Pi \to C$. The key note here is the fact
that, since only variables can appear under ${!}$, the contraction rule $(\CONTR)$ cannot
merge two connectives. Therefore, since the total number of connectives is less
than $n$, the number of logical rule applications is also less than $n$.

Each branching point corresponds to an application of a logical rule, whence the
number of branching points is also less than $n$. Therefore, in the tree there are
no more than $n$ axiom leaves, and each axiom introduces two variable occurrences.
Let's trace these occurrences down the tree. Each occurrence either traces to
an occurrence in the goal sequent, or disappears (gets merged with another occurrence)
in an application of $(\CONTR)$. Thus, the number of $(\CONTR)$ applications is
less than the total number of variable occurrences in axiom leaves, and, therefore,
less then $2n$.

Finally, we limit the number of $(\PERM_{1,2})$ applications. As said above,
a block of consecutive applications of $(\PERM_{1,2})$ can include an arbitrarily
large number of $(\PERM_{1,2})$ applications. However, we can always reduce it.
Each block of consecutive permutations has the following form:
$$
\infer[(\PERM_{1,2})^*,]
{\Delta'_1, {!}A_{i_1}, \Delta'_2, {!}A_{i_2}, \Delta'_3, \dots, \Delta'_k, {!}A_{i_k},
\Delta'_{k+1} \to B}
{\Delta_1, {!}A_1, \Delta_2, {!}A_2, \Delta_3, \dots, \Delta_k, {!}A_k, \Delta_{k+1} \to B}
$$
where the sequences 
$\Delta_1, \dots, \Delta_{k+1}$ and $\Delta'_1, \dots, \Delta'_{k+1}$ coincide
and $\{ i_1, \dots, i_k \} = \{1, \dots, k\}$.

The number of formulae in the left-hand side of the sequent here is bounded
by $3n$ (it was less than $n$ in the goal sequent $\Pi \to C$, and, in the
worst case, it was increased by less than $2n$ applications of $(\CONTR)$).
Therefore, $k < 3n$.
Now we replace this block with a block of $k$ permutations: each ${!}A_i$ is 
moved to its place by one permutation. Thus, in each block we have less than
$3n$ permutations. 

Each $(\PERM_{1,2})$ block is preceded by an application of a rule different
from $(\PERM_{1,2})$ or an axiom leaf. Thus
the number of such blocks is bounded by $4n$ ($n$ for logical rules,
$2n$ for contractions, $n$ for axioms). 

Therefore, the number of $(\PERM_{1,2})$ applications
is less than $12n^2$, and the total size of the derivation is less than
$12n^2 + 3n$. 
\end{proof}

This lemma yields the following decidability result:

\begin{theorem}\label{Th:NP}
The derivability problem in $\CC$ for sequents
in which ${!}$ is applied only to variables
is decidable and belongs to
the $\mathrm{NP}\!$ class
(i.e., can be resolved by
a nondeterministic polynomial algorithm).
\end{theorem}

\section{Future Work}

Since the Lambek calculus itself is NP-complete~\cite{Pentus2006}\cite{Savateev2009},
we get NP-completeness of $\CC$ in the restricted case,
where ${!}$ can be applied only to variables. On the other hand, it is known that
the derivability problem for the fragment of the pure Lambek calculus 
with only one division operation is
decidable in polynomial time~\cite{Savateev2008}\cite{PentusAiML}. 
Therefore, the complexity for the restricted case of $\CCu$ (where we
have only one division, and ${!}$ can be applied only to variables) yet should be 
studied. It belongs to NP (by our Theorem~\ref{Th:NP}), and the question 
is whether this fragment is poly-time decidable or
NP-hard.
Recall that in the unrestricted case we've proved
undecidability not only for the whole $\CC$, but also for its one-division fragment,
$\CCu$.

Another interesting question is whether our decidability result (Theorem~\ref{Th:NP})
can be extended to the situation where ${!}$ can 
be applied to formulae of Horn depth 1, {\em i.e.,} formulae,
in which all denominators of $\SL$ and $\BS$ are primitive types,
for instance, $(p \BS (q \SL r)) \SL s$.
Notice that if we allow formulae of Horn depth 2 (of the form
$r \SL (p \SL q)$) under ${!}$, then we immediately get undecidability
(see Section~\ref{S:undec}).

Our encoding in Theorem~\ref{Th:deduction} 
actually shows that {\em grammars} based on $\CC$ can generate
arbitrary recursively enumerable languages. On the other hand,
pure Lambek grammars generate precisely context-free 
languages~\cite{Pentus1997JSL}. Moreover, this holds also
in the so-called strong sense, {\em i.e.,} context-free grammars
and Lambek grammars can assign the same Montague-style semantic values to the
words derived~\cite{KanSalv}\cite{Kuzn2015TrMIAN}. The question
is what class of grammars in the Chomsky hierarchy corresponds to
grammars based on the fragment of $\CC$, restricted as in Theorem~\ref{Th:NP}, and
could one add Montague-style semantics to such grammars.

\subsection*{Acknowledgements}

Stepan Kuznetsov's research was supported by the Russian Foundation
for Basic Research 
(grants 15-01-09218-a and 14-01-00127-a)
and by the Presidential Council for Support of Leading
Scientific Schools
(grant N\v{S}-9091.2016.1). Max Kanovich's research was partially supported
by EPSRC. Andre Scedrov's research was partially supported by ONR.

This research was performed in part during visits of Stepan Kuznetsov and Max Kanovich
to the University of Pennsylvania. We greatly appreciate support of the
Mathematics Department of the University.
A part of the work was also done during the stay of Andre Scedrov at
the National Research University Higher School of Economics.
 We would like to thank 
S.~O.~Kuznetsov and I.~A.~Makarov for hosting there.

The paper was prepared in part within the framework of the Basic
Research Program at the National Research University Higher
School of Economics (HSE) and was partially supported within the
framework of a subsidy by the Russian Academic Excellence Project
`5--100'. 

We are indepted to the participants of the research seminars
``Logical Problems in Computer Science'' and ``Algorithmic Problems in
Algebra and Logic'' at Moscow (Lomonosov) University, in particular, S.~I.~Adian, L.~D.~Beklemishev,
V.~N.~Krupski, I.~I.~Osipov, F.~N.~Pakhomov, M.~R.~Pentus, D.~S.~Shamkanov, I.~B.~Shapi\-rovsky, 
V.~B.~Shehtman, A.~A.~Sorokin, T.~L.~Yavorskaya, and others for fruitful
discussions and suggestions that allowed us to improve our presentation
significantly.


\begin{thebibliography}{XX}
\bibitem{Abrusci90}
V. M. Abrusci. A comparison between Lambek syntactic calculus and
intuitionistic linear propositional logic. Zeitschr. f\"ur math. Logik und
Grundl. der Math. (Math. Logic Quart.), Vol.~36, 1990. P.~11--15.
\bibitem{Buszko1982} W. Buszkowski. Some decision problems
in the theory of syntactic categories. Zeitschr.
f\"ur math. Logik und Grundl. der Math. (Math. Logic
Quart.) Vol.~28, 1982. P. 539--548.
\bibitem{Buszko2005} W. Buszkowski. Lambek calculus with
nonlogical axioms. Language and Grammar (CSLI Lect.
Notes, vol.~168), 2005. P.~77--93.
\bibitem{Carpenter} B. Carpenter. Type-logical
semantics. MIT Press, 1998.
\bibitem{Girard87}
J.-Y. Girard. Linear logic. Theor. Comput. Sci., Vol.~50,
No.~1, 1987. P.~1--102.
\bibitem{KanSalv} M. Kanazawa, S. Salvati. The string-meaning relations
definable by Lambek grammars and context-free grammars. Proc. Formal Grammar
'12/'13 (LNCS vol.~8036), Springer, 2013. P.~191--208.
\bibitem{KanKuzSce} M. Kanovich, S. Kuznetsov,
A. Scedrov. On Lambek's restriction in the presence of
exponential modalities. Proc. LFCS '16 (LNCS
vol.~9537), Springer, 2015. P.~146--158.
\bibitem{Kuzn2015TrMIAN} S.~L.~Kuznetsov.
On translating context-free grammars into Lambek grammars. 
Proc. Steklov Inst. Math., Vol.~290, 2015. P.~63--69.
\bibitem{Lambek1958} J. Lambek. The mathematics of sentence
structure. Amer. Math. Monthly, Vol.~65, No.~3, 1958. P. 154--170.
\bibitem{Lambek1961} J. Lambek.
On the calculus of syntactic types.
Structure of Language and Its Mathematical Aspects
(Proc. Symposia Appl. Math., vol.~12), AMS, 1961.
P.~166---178.
\bibitem{LMSS} P. Lincoln, J. Mitchell, A. Scedrov, N. Shankar. Decision problems for propositional linear logic.
Annals of Pure and Applied Logic, Vol.~56, Iss.~1--3, 1992. P.~239--311.
\bibitem{MorVal} G. Morrill, O. Valent\'{\i}n. Computational coverage of TLG: Nonlinearity. Proc. NLCS '15 (EPiC Series, vol.~32), 2015. P.~51--63
\bibitem{Pentus1997JSL}
M. Pentus. Product-free Lambek calculus and context-free grammars.
J. Symb. Log., Vol.~62, No.~2, 1997. P.~648--660.
\bibitem{Pentus2006} M. Pentus.
Lambek calculus is NP-complete. Theor. Comput. Sci., Vol.~357, 2006.
P.~186--201.
\bibitem{PentusAiML}
M. Pentus. Complexity of the Lambek calculus and its fragments.
Advances in Modal Logic, vol.~8, College Publications, 2010. P.~310--329.
\bibitem{Savateev2008}
Yu. Savateev. Lambek grammars with one division
are decidable in polynomial time. Proc. CSR '08 (LNCS vol.~5010),
Springer, 2008. P.~273--282.
\bibitem{Savateev2009}
Yu. Savateev. Product-free Lambek calculus is NP-complete.
Proc. LFCS '09 (LNCS vol.~5407), Springer, 2009. P.~380--394.
\bibitem{Yetter90}
D. N. Yetter. Quantales and (noncommutative) linear logic. J. Symb. Log. 
Vol.~55, No.~1, 1990. P.~41--64.
\end{thebibliography}
\end{document}